\documentclass{amsart}
\usepackage{amssymb,latexsym}
\usepackage{graphicx}
\theoremstyle{plain}
\newtheorem{theorem}{Theorem}[section]
\newtheorem{corollary}[theorem]{Corollary}

\newtheorem{lemma}[theorem]{Lemma}
\newtheorem{proposition}[theorem]{Proposition}

\theoremstyle{definition}

\newtheorem*{remark}{Remark}
\theoremstyle{remark}

\newtheorem{example}[theorem]{Example}

\begin{document}
    \title{Products and Factors of Banach Function spaces}
\author{Anton~R. Schep}
\address{Department of Mathematics\\
         University of South Carolina\\
	 Columbia, SC 29208}
\email{schep@math.sc.edu}	  
\begin{abstract} 
Given two Banach function spaces we study the pointwise product  space $E\cdot F$, especially for the case that the pointwise product of their unit balls is again convex. We then give conditions on when the pointwise product $E\cdot M(E, F)=F$ , where $M(E,F)$ denotes the space of multiplication operators from $E$ into $F$.
\end{abstract}
\maketitle
\section*{Introduction.}
Let $(X,\Sigma, \mu)$ be a complete $\sigma$-finite measure space. By $L_{0}(X, \mu)$ we will denote the set of all measurable functions which are finite a.e.. As usual we will identify functions equal almost evrywhere. A linear subspace of $L_{0}(X, \mu)$ is called a {\it K\"othe function space} if it is normed space which is an order ideal in  $L_{0}(X, \mu)$, i.e,, if $f\in E$ and $|g|\le |f|$ a.e., then $g\in E$ and $\|g\|\le \|f\|$. A norm complete K\"othe function space is called a {\it Banach function space}. Given a  K\"othe function space $E$ we have that 
\[E^{\prime}=\{f\in L_{0}: \int_{X}|fg|\,d\mu <\infty \mbox{ for all }g\in E\}\]
is a Banach function space with the Fatou property, which is called the associate space of $E$.
For a detailed treatment of Banach function spaces we refer to \cite{Zaanen67}. The detailed study of Banach function spaces led to the the study of Riesz spaces and Banach lattices, which incorporated, clarified and extended the earlier theory, see e.g. \cite{Zaanen83}. In this paper we wil study the pointwise product   $E\cdot F=\{f\cdot g: f\in E, g\in F\}$ of two Banach function spaces $E$ and $F$. The main question we will be interested in is whether $E\cdot F$ is again a Banach function space. The finite dimensional case, i.e. the case that $X$ consists of finitely many atoms, shows that one must impose some additional requirements to make this an interesting question. In this paper we will be requiring that the pointwise product $B_{E}\cdot B_{F}$ of the respective unit balls of $E$ and $F$ is again the unit ball of a Banach function space, in which case we will call $E\cdot F$ a product Banach function space. In fact, one of the motivations of this paper was a result in a paper of  Bollobas and Leader in \cite{Bollobas93} (see also \cite{Bollobas95}), who studied the pointwise product of unconditional convex bodies in $\mathbb R^{n}$, which can viewed as a finite dimensional version of one of our results. The main examples of product Banach function spaces are given on the one hand by the fundamental result of Lozanovskii (\cite{Lozanovskii69}, see also \cite{Gillespie81}) which states  that for any Banach function space $E$ the product $E\cdot E^{\prime}$ is a product Banach function space isometrically equal to $L_{1}(X, \mu)$, and on the other hand by the Calderon intermediate space  $E^{\frac 1p}F^{\frac 1{p^{\prime}}}$ which  is a product Banach function space for any pair of Banach function spaces and $1<p<\infty$. In some sense the present paper can be viewed as providing the tools to show that most (if not all) examples in the literature of product Banach function spaces can be deduced from one of these two examples of product Banach function spaces. The present paper is organized as follows. In section 1 we collect some, mostly elementary, results about the normability of the pointwise product of Banach function spaces and show that this is equivalent to having sufficiently many multiplication operators from $E$ into $F^{\prime}$. This leads us to consider, for any pair $E$ and $F$  of Banach function spaces,  the Banach function space of multiplication operators from $E$ into $F$ which we denote by $M(E, F)$, i.e., 
\begin{equation*}
M(E,F)=\{g\in L_{0}(X, \mu): fg\in F \text{ for all } f\in E \},
\end{equation*}
and the norm on $M(E, F)$ is the operator norm
\begin{equation*}
\|g\|_{M(E, F)}=\sup\{\|fg\|_{F}: \|f\|_{E}\le 1\}.
\end{equation*}
We give for non-atomic measures a result which shows that $M(E, F)=\{0\}$ in case the upper index $\sigma(E)$ is less than the lower index $s(F)$ of $F$. This generalizes the well-known result that $M(L^{p}, L^{q})=\{0\}$, whenever $1\le p <q\le \infty$ and $\mu$ is non-atomic. In section 2 we study the basic properties of product Banach function spaces. As the reader will see these results depend heavily on the above mentioned factorization theorem of Lozanovskii. One of the most important results is a cancellation result which says that if  $E, F$ and $G$ are Banach function spaces with the Fatou property, and if we  assume that $E\cdot F$ and $E\cdot G$ are product Banach function spaces such that $E\cdot F\subset E\cdot G$  with $\|h\|_{E\cdot G}\le C\|h\|_{E\cdot F}$ for all $h\in E\cdot F$,  then $F\subset G$ and $\|f\|_{G}\le C \|f\|_{F}$ for all $f\in F$. In section 3 we  consider the problem of division. We will call the Banach function space $E$ a factor of the Banach function space $G$ if there exist a Banach function space $F$ such that $E\cdot F=G$.  Lozanovskii's factorization theorem says that every Banach function space is a factor of $L_{1}(X, \mu)$. To get uniqueness of factors this leads to the question whether $E\cdot M(E, G)$ is a product Banach function space and whether it is equal to $G$.
One of the main results we  prove is that if $E$ and $G$ are Banach function spaces such that there exists $1<p<\infty$ such that $E$ is $p$-convex and $G$ is $p$-concave with convexity and concavity constants equal to 1 and if $E$ has the Fatou property, then ] $E\cdot M(E, F)$ is a product Banach function space and $E\cdot M(E, F)=F$ and $E=M(M(E, F), F)$. In section 4 we presnt some application of our results. In \cite{Bennett96} G. Bennett showed that many classical inequalities involving the $\ell_{p}$-norm can be expressed as a result of product Banach sequence spaces. We will show that some of his results are easy consequences of the results of the previous sections.

	The results of this paper and of \cite{Schep2008} were announced at the V-th Positivity conference in July 2007 in Belfast. In March 2008 we learned about the preprint \cite{Calabuig2008}, where some of the results of this paper are duplicated independently of this paper.

\section{The normed product space}Let  $E$ and $F$ be Banach function spaces on $(X,\Sigma, \mu)$. We will assume that both $E$ and $F$ are saturated Banach function spaces. In this section we discuss when $E\cdot F=\{f\cdot g: f\in E, g\in F\}$ is a normed K\"othe space. The norm, if it exists, will always assumed to be generated by the convex hull of the pointwise product of the unit balls $B_{E}\cdot B_{F}$. In this case one can easily verify that for all $h\in E\cdot F$ we have 
\begin{equation*}
\|h\|_{E\cdot F}=\inf\left\{\sum_{k=1}^{n}\|f_{k}\|_{E}\|g_{k}\|_{F}: |h|\le \sum_{k=1}^{n }f_{k}g_{k}, 0\le f_{k}\in E, 0\le g_{k}\in F \right\}.
\end{equation*}
In case the above expression defines a norm we will say that $E\cdot F$ is normable. From the discussion in \cite{schep2005} and the fact that $(E\cdot F^{)\prime}$ contains a strictly positive element in case $E\cdot F$ is normable, we get immediately the following proposition.
\begin{proposition} Let  $E$ and $F$ be Banach function spaces on $(X,\Sigma, \mu)$. Then the following are equivalent. 
\begin{itemize}
\item [(i)] $E\cdot F$ is normable.
\item[(ii)] There exists $0<g\in L_{0}$ such that $g\cdot F\subset E^{\prime}$.
\item[(iii)] There exists $0<g\in L_{0}$ such that $g\cdot E\subset F^{\prime}$.
\item[(iv)] There exist disjoint measurable sets $X_{n}$ with $\cup_{n}X_{n}=X$ such that $F_{|X_{n}}\subset E^{\prime}$ for all $n$.
\item[(v)] There exist disjoint measurable sets $X_{n}$ with $\cup_{n}X_{n}=X$ such that $E_{|X_{n}}\subset F^{\prime}$ for all $n$.
\end{itemize}
\end{proposition}
As condition (iv) (or (v)) holds automatically when the measure $\mu$ is atomic (we can take $X_{n}$ to be an atom for all $n$),  we have the following corollary.
\begin{corollary}
Let  $E$ and $F$ be Banach function spaces on $(X,\Sigma, \mu)$ and assume $\mu$ is an atomic measure. Then $E\cdot F$ is normable.
\end{corollary}
In general it is not true that $E\cdot F$ is complete, whenever $E\cdot F$ is normable, as can be seen from the following simple example.
\begin{example}
Let $E=F=\ell_{1}$. Then $\ell_{1}\cdot \ell_{1}$ as a set is equal to $\ell_{\frac 12}$, but $\|\cdot\|_{E\cdot F}=\|\cdot\|_{1}$, so that $(E\cdot F, \|\cdot\|_{E\cdot F})$ is not complete. 
\end{example}
In this example and more generally, when $E\cdot F$ is normable, then the completion of $(E\cdot F, \|\cdot \|_{E\cdot F})$ is  the Banach envelope of the quasi-normed space $E\cdot F$, provided with the quasi-norm 
\begin{equation*}
\inf \{\|g\|_{E} \|h\|_{F}: |f|=gh, 0\le g\in E, 0\le h\in F\}.
\end{equation*}
We refer to \cite{Kalton84} for some general remarks about the Banach envelope of locally bounded space with separating dual.
For any pair $E$ and $F$  of Banach function spaces we denote by $M(E, F)$ the Banach function space of multiplication operators from $E$ into $F$, i.e., 
\begin{equation*}
M(E,F)=\{g\in L_{0}(X, \mu): fg\in F \text{ for all } f\in E \},
\end{equation*}
and the norm on $M(E, F)$ is the operator norm
\begin{equation*}
\|g\|_{M(E, F)}=\sup\{\|fg\|_{F}: \|f\|_{E}\le 1\}.
\end{equation*}
 Note that it can happen (as we will see in more detail later on), that $M(E, F)=\{0\}$. The following proposition relates the normability of $E\cdot F$ to the saturation of $M(E, F^{\prime})$.
 \begin{proposition}\label{dual}
Let  $E$ and $F$ be Banach function spaces on $(X,\Sigma, \mu)$. Then the following are equivalent.
\begin{itemize}
\item[(i)] $E\cdot F$ is normable.
\item[(ii)] The Banach function space $M(E, F^{\prime})$ is saturated.
\end{itemize}
In case $E\cdot F$ is normable, then $(E\cdot F)^{\prime}=M(E, F^{\prime})=M(F, E^{\prime})$ isometrically.
\end{proposition}
\begin{proof}
The equivalence of (i) and (ii) is just a rephrasing of the equivalence of (i) and (iii) of the above proposition. Assume now that $E\cdot F$ is normable and let $0\le h\in (E\cdot f)^{\prime}$ with $\|h\|_{(E\cdot F)^{\prime}}\le 1$. Then let $0\le f\in E$ and $0\le g\in F$. Then $fg\in E\cdot F$ with $\|fg\|_{E\cdot F}\le \|f\|_{E}\|g\|_{F}$. Hence $0\le \int (hf)g\,d\mu\le \|f\|_{E}\|g\|_{F}$. This implies that $hf\in F^{\prime}$ and $\|hf\|_{F^{\prime}}\le \|f\|_{E}$, i.e., $h\in M(E, F^{\prime})$ and $\|h\|_{M(E, F)}\le 1$. This shows $(E\cdot F)^{\prime}\subset M(E, F^{\prime})$ and $\|h\|_{M(E, F)}\le \|h\|_{(E\cdot F^{)\prime}}$. Now let $0\le h\in M(E, F^{\prime})$ with $\|h\|_{M(E, F^{\prime})}\le 1$ and let $0\le \tilde{h}\in E\cdot F$ with $\|\tilde{h}\|_{E\cdot F}<1$. Then there exist $0\le f_{i}\in E$ and $g_{i}\in F$ with $0\le \tilde{h}\le \sum_{i=1}^{n}f_{i}g_{i}$ and $\sum_{i=1}^{n}\|f_{i}\|_{E}\|g_{i}\|_{F}<1$. This implies
\begin{equation*}
\int \tilde{h}h\,d\mu\le \sum_{i=1}^{n}\int (hf_{i})g_{i}\,d\mu\le\sum_{i=1}^{n}\|hf_{i}\|_{F^{\prime}}\|g_{i}\|_{F}\le \sum_{i=1}^{n}\|f_{i}\|_{E}\|g_{i}\|_{F}<1.
\end{equation*}
Hence $h\in (E\cdot F)^{\prime}$ and $\|h\|_{(E\cdot F)^{\prime}}\le 1$. This shows $M(E, F^{\prime})\subset (E\cdot F)^{\prime}$ and $ \|h\|_{(E\cdot F^{)\prime}}\le \|h\|_{M(E, F)}$, which completes the proof  that $(E\cdot F)^{\prime}=M(E, F^{\prime})$. Since $E\cdot F=F\cdot E$, it also follows that $(E\cdot F)^{\prime}=M(F,  E^{\prime})$.
\end{proof}
\begin{corollary}
Let  $E$  be a Banach function space on $(X,\Sigma, \mu)$. Then $E\cdot E$ is normable if and only if there exists $0<h\in L_{0}$ such that $E\subset L_{2}(X, hd\mu)$.
\end{corollary}
\begin{proof} If $E\cdot E$ is normable, then $(E\cdot E)^{\prime}=M(E, E^{\prime})$ contains a strictly positive $h$. Then $\int f^{2}hd\mu=\int f (fh)\,d\mu\le \|f\|_{E}\|fh\|_{E^{\prime}}<\infty$ for all $f\in E$, i.e., 
$E\subset L_{2}(X, hd\mu)$. Conversely, if $E\subset L_{2}(X, hd\mu)$ for some strictly positive $h$, then by Cauchy-Schwarz's inequality we have that $\int |f_{1}f_{2}|h\,d\mu<\infty$ for all $f_{1}, f_{2}\in E$, which shows that $hE\subset E^{\prime}$, so that  $E\cdot E$ is normable.
\end{proof}

We will now show that in certain cases we have $M(E, F)=\{0\}$. First we state a simple lemma.
\begin{lemma} \label{simple}Let $a,b, c, d, e, f$ be positive real numbers, such that $a\le b+c$ and $d\ge e+f$. Then
\[\frac {a}{d}\le \max \left\{\frac be, \frac cf\right\}.\]
\end{lemma}
\begin{proof} From the assumption it follows immediately that $\frac ad\le \frac {b+c}{e+f}$.
Assume $\frac cf\le \frac be$. Then $\frac {b+c}{e+f}\le \frac be$. Hence $\frac ad\le \frac {b+c}{e+f}\le \frac be$.
\end{proof}
Recall now that a Banach lattice $E$ 
 Similarly a Banach lattice $E$ is called
$p$--convex for $1\le p \le\infty$ if there exists a constant $M$ such that
for all $f_1,\dots, f_n\in E$ 
$$\left\|\left(\sum^n_{k=1}|f_k|^{p} \right)^{\frac 1p} \right\|_{E}\le
M\left(\sum^n_{k=1}\|f_k \|^p_{E}\right)^{\frac 1p} \text{ if }1\le p<\infty $$
or $\|(\sup |f_k|\|_{E}\le M \max_{1\le k\le n}\|f_k\|_{E}$ if $p=\infty.$
Similarly $E$ is called
$p$--concave for $1\le p \le\infty$ if there exists a constant $M$ such that
for all $f_1,\dots, f_n\in E$ 
$$\left(\sum^n_{k=1}\|f_k \|^p_{E}\right)^{\frac 1p} \le M\left\|\left(\sum^n_{k=1}|f_k|^{p} \right)^{\frac 1p} \right\|_{E}\text{ if }1\le p<\infty$$ 
or $\max_{1\le k\le n}\|f_k\|_{E}\le M\|\sup |f_k|\|_{E} $ if $p=\infty.$
The notions of $p$--convexity, respectively $p$--concavity are
closely related to the notions of upper $p$--estimate (strong
$\ell_p$--composition property), respectively lower $p$--estimate (strong
$\ell_p$--decomposi\-tion property) as can be found in e.g. \cite[ Theorem
1.f.7]{Lindenstrauss79}.
Then the numbers
$$\sigma(E)= \inf (p\ge 1: E\text{ is $p$-concave })$$
and
$$s(E)=\sup (p\ge 1: E\text{ is $p$-convex })$$
are called the {\bf upper} or {\bf lower index} of $E$, respectively.
If $\operatorname {dim}(E)= \infty$, then $1\le s(E)\le \sigma(E)\le \infty$.
We collect some basic facts about the indices of a Banach lattice: If
$\sigma(E)<\infty$, then $E$ has order continuous norm and if $s(E)>1$, then the
dual space $E^*$ has order continuous norm. Also we have
$$\frac 1{s(E)} + \frac 1{\sigma (E^*)}=1 \text{ and } \frac 1{\sigma(E)}
+\frac 1{s(E^*)}=1.$$
In case $E$ is a Banach function space with the (weak) Fatou property, then we have $s(E^{*})=s(E^{\prime})$ and $\sigma(E^{*})=\sigma(E^{\prime})$. For this and additional details see \cite{Grobler75} and \cite{Dodds77}.
\begin{lemma}
Let  $E$ be a Banach function space on $(X,\Sigma, \mu)$. Assume $E$ has order continuous norm and $\mu$ is non-atomic. Then for all $0\le f\in E$ there exist $0\le g\le f\in E$ with $g\wedge (f-g) =0$  such that $\|g\|_{E}=\|f-g\|_{E}$.
\end{lemma}
\begin{proof} First observe that for all $0\le f \in E$ with $f\neq 0$ we can find for all $\epsilon>0$ a component $0\le g\le f\in E$ with $g\wedge (f-g) =0$  such that $0<\|g\|_{E}<\epsilon$. In fact we can find a sequence $X_{n}\downarrow \emptyset$ with $0<\mu(X_{n})<\frac 1n$. Then $\|f\chi_{X_{n}}\|_{E}\downarrow 0$ implies that there exists $n_{0}$ such that $\|f\chi_{X_{n_{0}}}\|_{E}<\epsilon$. Now let $0\le f\in E$ and consider the set $\mathcal P=\{g\in E: 0\le g\le f, g\wedge (f-g)=0, \|g\|_{E}\le \|f-g\|_{E}\}$.  Then $\mathcal P\neq \emptyset$ and with the ordering inherited from $E$ it has the property that every chain in $\mathcal P$ has a least upper bound, by the order continuity of the norm. Hence $\mathcal P$ has a maximal element $g_{0}$. Assume $\|g_{0}\|_{E}< \|f-g_{0}\|_{E}$. Then take $\epsilon=\frac 12 (\|f-g_{0}\|_{E}-\|g_{0}\|_{E})$. By the remark in the beginning of the proof we can a component $0<g_{1}\le (f-g_{0})$ of $f-g_{0}$ such that $\|g_{1}\|_{E}<\epsilon$. One can verify now easily that $g_{0}+g_{1}\in \mathcal P$, which contradicts the maximality of $g_{0}$. Hence $\|g_{0}\|_{E}=\|f-g_{0}\|_{E}$ and the proof is complete.
\end{proof}
\begin{theorem}
Let  $E$ and $F$ be Banach function spaces on $(X,\Sigma, \mu)$ and assume $\mu$ is non-atomic. 
Then $\sigma(E)<s(F)$ implies that $M(E, F)=\{0\}$.
\end{theorem}
\begin{proof}
Let $ \sigma(E)<p<q<s(F)$. Then $E$ is $p$-concave and $F$ is $q$-convex. Renorming $E$ and $F$, if necessary, we can assume that the concavity and convexity constants are 1. Also $p<\infty$ implies that the space $E$ has order continuous norm. Let now $0\le h\in M(E, F)$ and $0\le f_{0}\in E$ with $\|f_{0}\|_{E}>0$. By the above lemma we can write $f_{0}=g_{0}+h_{0}$ with $g_{0}\wedge h_{0}=0$ and such that $\|g_{0}\|_{E}=\|h_{0}\|_{E}$. By the p-concavity we have $\|f\|_{E}^{p}\ge \|g_{0}\|_{E}^{p}+\|h_{0}\|_{E}^{p}=2\|g_{0}\|_{p}$ and by the $q$-convexity of $F$ we have $\|hf_{0}\|_{F}^{q}\le \|hg_{0}\|_{F}^{q}+\|hh_{0}\|_{F}^{q}$. It follows now from Lemma \ref{simple} that there exists $f_{1}$ equal to either $g_{0}$ or $h_{0}$ such that
\begin{equation*}
\frac {\|hf_{1}\|_{F}^{q}}{\|f_{1}\|_{E}^{p}}\ge \frac {\|hf_{0}\|_{F}^{q}}{\|f_{0}\|_{E}^{p}}.
\end{equation*}
By induction we can now find $f_{n+1}\ge 0$ such that $\|f_{n}\|_{E}^{p}\ge 2 \|f_{n+1}\|^{p}$ and 
\begin{equation*}
\frac {\|hf_{n+1}\|_{F}^{q}}{\|f_{n+1}\|_{E}^{p}}\ge \frac {\|hf_{n}\|_{F}^{q}}{\|f_{n}\|_{E}^{p}}.
\end{equation*}
Assume now that $hf_{0}\neq 0$. Then we have
\begin{align*} 
\|h\|_{M(E, F)}^{q}&\ge  \frac {\|hf_{n}\|_{F}^{q}}{\|f_{n}\|_{E}^{q}}= \frac {\|hf_{n}\|_{F}^{q}}{\|f_{n}\|_{E}^{p}}\cdot \frac 1{\|f_{n}\|_{E}^{q-p}}\\
	&\ge  \frac {\|hf_{0}\|_{F}^{q}}{\|f_{0}\|_{E}^{q}}\cdot 2^{\frac {n(q-p)}p}\cdot \frac 1{\|f_{0}\|_{E}^{q-p}}\to \infty,
\end{align*}
which is contradiction. Hence $hf_{0}=0$ for all $f_{0}\in E$ and thus $h=0$. 
\end{proof}
\begin{corollary} Let  $E$ and $F$ be Banach function spaces on $(X,\Sigma, \mu)$ and assume $\mu$ is non-atomic. Then $\frac 1{\sigma(E)}+\frac 1{\sigma(F)}>1$ implies that $E\cdot F$ is not normable.
\end{corollary}
\begin{proof}
From $\frac 1{\sigma(E)}+\frac 1{\sigma(F)}>1$, it follows that $\frac 1{\sigma(E)}>1-\frac 1{\sigma(F)}=\frac 1{s(F^{\prime})}$.  Hence by the above proposition $M(E, F^{\prime})=\{0\}$.
\end{proof}
\section{Product Banach function spaces}
As seen from the examples in the previous section it can happen that $E\cdot F$ is a normed K\"othe space, but is not complete. To better understand when $E\cdot F$ is a Banach function space we need the Calderon construction of intermediate spaces, which was studied and extended by Lozanovskii.
Let $E$ and $F$ be Banach function spaces on $(X, \mu)$. Then for $1< p<\infty$ the Banach function space $E^{\frac 1p}F^{\frac 1{p^{\prime}}}$ is defined as the space of all $f\in L_{0}$ such that $|f|=|g|^{\frac 1p}|h|^{\frac 1{p^{\prime}}}$ for some $g\in E$ and $h\in F$. The norm on $E^{\frac 1p}F^{\frac 1{p^{\prime}}}$ is defined by
\begin{align*}
\|f\|_{E^{\frac 1p}F^{\frac 1{p^{\prime}}}}&=\inf \{\|g\|_{E}^{\frac 1p}\|h\|_{F}^{\frac 1{p^{\prime}}}: |f|=|g|^{\frac 1p}|h|^{\frac 1{p^{\prime}}} \text{ for some }g\in E, h\in F\}\\
&=\inf \{\lambda: |f|=\lambda |g|^{\frac 1p}|h|^{\frac 1{p^{\prime}}}, \|g\|_{E}\le 1, \|h\|_{F}\le 1\}.
\end{align*}
It is well-known that $E^{\frac 1p}F^{\frac 1{p^{\prime}}}$ is again a Banach function space, moreover it has order continuous norm if at least one of $E$ and $F$ has order continuous norm. Also $E^{\frac 1p}F^{\frac 1{p^{\prime}}}$ has the Fatou property if both $E$ and $F$ have the Fatou property. The above construction contains as a special case the so-called $p$-convexication of $E$ by taking $F=L_{\infty}$. We will denote this space by $E^{\frac 1p}$, since as sets $E^{\frac 1p}L_{\infty}^{\frac 1{p^{\prime}}}=E^{\frac 1p}$. Note that the norm on $E^{\frac 1p}$ is given by $\||f|^{p}\|_{E}^{\frac 1p}$ for all $f\in E^{\frac 1p}$. This implies that the space $E^{\frac 1p}F^{\frac 1{p^{\prime}}}=E^{\frac 1p}\cdot F^{\frac 1{p^{\prime}}}$ as defined in the previous section and that
\begin{equation*}
\|f\|_{E^{\frac 1p}\cdot F^{\frac 1{p^{\prime}}}}=\inf \{\|g\|_{E^{\frac 1p}}\|h\|_{F^{\frac 1{p^{\prime}}}}: |f|=|g|\cdot|h|, g\in E^{\frac 1p}, h\in F^{\frac 1{p^{\prime}}}\}.
\end{equation*}
In particular the pointwise product of the unit balls $B_{E^{\frac 1p}}$ and $B_{{F^{\frac 1{p^{\prime}}}}}$ is convex. Let again $E$ and $F$ be Banach function spaces on $(X,\Sigma, \mu)$. Assume that $E\cdot F$ is a normable K\"othe function space. Then we will say that $E\cdot F$ is a {\it product Banach function space} if $E \cdot F$ is complete and the norm on $E\cdot F$ is given by
\begin{equation*}
\|f\|_{E\cdot F}=\inf \{\|g\|_{E} \|h\|_{F}: |f|=gh, 0\le g\in E, 0\le h\in F\},
\end{equation*}
i.e., the pointwise product $B_{E}\cdot B_{F}$ of the unit balls of $E$, respectively $F$, is convex. From the above discussion it is clear that $E^{\frac 1p}\cdot F^{\frac 1{p^{\prime}}}$ is a product Banach function space. Also the fundamental result of Lozanovskii (\cite{Lozanovskii69}, see also \cite{Gillespie81}) is that for any Banach function space $E$ the product $E\cdot E^{\prime}$ is a product Banach function space isometrically equal to $L_{1}(X, \mu)$. 
We  first show that completeness can  be omitted from our definition of product Banach function space. Then we will recall from \cite{Schep2008} that more generally the pointwise product $B_{E}\cdot B_{F}$ of the unit balls $B_{E}$ and $B_{F}$ is closed with respect to a.e. convergence in case the norms on $E$ and $F$ have the Fatou property. Recall that if $E\cdot F$ is a normable K\"othe space, then for $f\in E\cdot F$ the quasi-norm of $f\in E\cdot F$ is given by
\begin{equation*}
\rho_{E\cdot F}(f)=\inf \{\|g\|_{E}\|h\|_{F}: |f|=gh, 0\le g\in E, 0\le h\in H\}. 
\end{equation*}
Note that $\|f\|_{E\cdot F}(f)\le \rho_{E\cdot F}(f)$  and  equality holds if and only if  $B_{E}\cdot B_{F}$ is convex. In the next theorem we will need the $p$-concavification of a Banach function space. If $E$ is $p$-convex for some $p>1$, then we can define the $p$-concavification $E^{p}$ by $f\in E^{p}$ if $f\in L_{0}(X, \mu)$ such that $|f|^{p}\in E$. If the convexity constant is equal to one, then $\|f\|_{E^{p}}=\|\,|f|^{p}\|_{E}^{\frac 1p}$ is a norm on $E^{p}$ and $E^{p}$ is complete with respect to this norm (see \cite{Lindenstrauss79}).
\begin{theorem}
Let  $E$ and $F$ be Banach function spaces  and assume that $\rho_{E\cdot F}$ is a norm. Then $E\cdot F$ is complete with respect to $\rho_{E\cdot F}$, i.e. $E\cdot F$ is a product Banach function space.
\end{theorem}
\begin{proof}
Let $G=E^{\frac 12}F^{\frac 12}$. Then $G$ is a Banach function space and $G$ is the 2-convexification of the normed lattice $E\cdot F$ with the norm $\rho_{E\cdot F}$. In particular $G$ is 2-convex with convexity constant equal to one. Now $E\cdot F$ is isometric to the 2-concavification of $G$ and thus $E\cdot F$ is complete.
\end{proof}
\begin{corollary}Let  $E$ and $F$ be Banach function spaces. Then $E\cdot F$ is a product Banach function space if and only if $E^{\frac 12}F^{\frac 12}$ is 2-convex with convexity constant one. In particular  $E\cdot E$ is a product Banach function space if and only if $E$ is 2-convex with convexity constant one. 
\end{corollary}
From \cite{Schep2008} we have the following result.
\begin{theorem}
Let  $E$ and $F$ be Banach function spaces with the Fatou property. Then $B_{E}\cdot B_{F}$ is closed in $L_{0}(X, \mu)$ with respect to a.e. convergence.  If addition  $\rho_{E\cdot F}$ is a norm, then $E\cdot F$ is a product Banach function space with the Fatou property.
\end{theorem}

The following theorem says  for product Banach function spaces defined  by Banach function spaces withe the Fatou property that the infimum in the definition of the norm is actually attained.

\begin{theorem}\label{minimum} Let  $E$ and $F$ be Banach function spaces with the Fatou property and assume that $E\cdot F$ is a product Banach function space. Then for all $0\le f\in E\cdot F$ there exist $0\le g\in E$ and $0\le h\in F$ such that $f=gh$ and $\|f\|_{E\cdot F}=\|g\|_{E}\|h\|_{F}$.
\end{theorem}
\begin{proof} Let $0\le f\in E\cdot F$ with $\|f\|_{E\cdot F}=1$. Then there exist $0\le g_{n}\in E$, $0\le h_{n}\in F$ with $f=g_{n}h_{n}$ and $\|g_{n}\|_{E}\le 1+\frac 1{2^{n}}$, $\|h_{n}\|_{F}\le 1+ \frac 1{2^{n}}$ for all $n\ge 1$. From Koml\'os' theorem for Banach function spaces and a theorem on products of Cesaro convergent sequences  (see \cite{Schep2008}) it follows that  there exist subsequences $\{g_{n_{k}}\}$ and $\{h_{n_{k}}\}$ and $g\in E$ with $\{g_{n_{k}}\}$ Ces\`aro converges a.e to $g$ and $\{h_{n_{k}}\}$ Ces\`aro converges a.e to $h$  such that $f\le gh$. Replacing $h$ by a smaller function we can assume $f=gh$.  Moreover 
\begin{equation*}
\|\frac 1k(g_{n_{1}}+\cdots g_{n_{k}})\|_{E}\le \frac 1k(1+\frac 1{2^{n_{1}}}+\cdots +1+\frac 1{2^{n_{k}}})\le 1+\frac 1k
\end{equation*}
implies that $\|g\|_{E}\le 1$. Similarly $\|h\|_{F}\le 1$. As $\|f\|_{E\cdot F}\le \|g\|_{E}\|h\|_{F}$ this implies that $\|g\|_{E}= 1$ and $\|h\|_{F}= 1$. 

\end{proof}
\begin{theorem}
Let $E, F$ and $G$ be Banach function spaces with the Fatou property. Assume that $E\cdot F$ and $E\cdot G$ are product Banach function spaces such that $E\cdot F\subset E\cdot G$  with $\|h\|_{E\cdot G}\le C\|h\|_{E\cdot F}$ for all $h\in E\cdot F$. Then $F\subset G$ and $\|f\|_{G}\le C \|f\|_{F}$ for all $f\in F$.
\end{theorem}
\begin{proof}
If $0\le f\in F$, then there exist $0\le f_{n}\uparrow f$ such that $f_{n}\in F\cap G$. Hence if $\|f\|_{G}\le C \|f\|_{F}$ for all $f\in F\cap G$, then the same inequality holds by the Fatou property for all $f\in F$. Assume therefore that there exists $0<f\in F\cap G$ such that $\|f\|_{G}>C\|f\|_{F}$. By normalizing we can assume that $\|f\|_{F}=1$ and thus $\|f\|_{G}>C$. Then there exists $0\le g\in G^{\prime}$ such that $C_{1}=\int fg\,d\mu>C$ and $\|g\|_{G^{\prime}}\le 1$. Now $\frac 1{C_{1}}fg\in L^{1}$ with $\|\frac 1{C_{1}}fg\|_{1}=1$, so by Lozanovskii's theorem there exist $0\le f_{1}\in E$ with $\|f_{1}\|_{E}\le 1$ and $0\le g_{1}\in E^{\prime}$ with $\|g_{1}\|_{E^{\prime}}\le 1$ such that $\frac 1{C_{1}}fg=f_{1}g_{1}$.  It follows now that
\begin{equation*}
\int f^{\frac 12}g^{\frac 12}f_{1}^{\frac 12}g_{1}^{\frac 12}\,d\mu=\frac 1{C_{1}^{\frac 12}}\int fg\,d\mu=C_{1}^{\frac 12}.
\end{equation*}
On the other hand $f_{1}f\in E\cdot F$ with $\|f_{1}f\|_{E\cdot F}\le \|f_{1}\|_{E}\|f\|_{F}\le 1$ implies that $\|f_{1}f\|_{E\cdot G}\le C$. Hence for $C<C_{2}<C_{1}$ there exist $f_{2}\in E$ with $\|f_{2}\|_{E}\le 1$ and $g_{2}\in G$ with $\|g_{2}\|_{G}\le C_{2}$ such that $f_{1}f=f_{2}g_{2}$. This implies that 
\begin{align*}
\int f^{\frac 12}g^{\frac 12}f_{1}^{\frac 12}g_{1}^{\frac 12}\,d\mu=&\int g_{2}^{\frac 12}g^{\frac 12}g_{1}^{\frac 12}f_{2}^{\frac 12}\,d\mu\le \left(\int g_{2}g\,d\mu\right)^{\frac 12} \left(\int g_{1} f_{2\,d\mu}\right)^{\frac 12} \\
&\le \|g_{2}\|_{G}^{\frac 12}\|g\|_{G^{\prime}}^{\frac 12}\|g_{1}\|_{E^{\prime}}^{\frac 12}\|f_{2}\|_{E}^{\frac 12}\le C_{2}^{\frac 12}<C_{1}^{\frac 12},
\end{align*}
which is a contradiction.
\end{proof}
\begin{corollary}\label{cancellation}
Let $E, F$ and $G$ be Banach function spaces with the Fatou property. Assume that $E\cdot F$ and $E\cdot G$ are product Banach function spaces with $E\cdot F=E\cdot G$ isomorphically (or isometrically). Then $F=G$ isomorphically (or isometrically).\end{corollary}
Note that the above corollary is no longer true if we drop the assumption of the Fatou property, we have e.g. that $\ell_{1}\cdot c_{0}=\ell_{1}\cdot \ell_{\infty}=\ell_{1}$. However the same proofs would give that $F^{\prime \prime}\subset G^{\prime\prime}$ in the above theorem and $F^{\prime\prime}=G^{\prime\prime}$ in the above corollary, if we omit the hypothesis of the Fatou property for the spaces $F$ and $G$. Moreover in the above theorem and corollary we can omit the hypothesis that $E$ has the Fatou property by using Lozanovskii's approximate factorization (this introduces a $1+\epsilon$ term in the factorization). The proof of the above theorem was inspired by a proof of \cite{Berezhnoi05}, who  used it to give an alternative proof of the uniqueness theorem of  the Calderon-Lozaonovskii interpolation method as given in \cite{Cwikel03}.     The following proposition provides a strengthening   of the above corollary and seems to be new even for Calderon spaces.
\begin{proposition}\label{uniqueness} Let $E_{1}, E_{2}, F_{1}$ and $F_{1}$ be Banach function spaces
with the Fatou property. Assume that $E_{1}\cdot F_{1}$ and $E_{2}\cdot F_{2}$ are product Banach function spaces such that $E_{1}\cdot F_{1}=E_{2}\cdot F_{2}$ isomorphically (or isometrically) and $E_{1}\subset E_{2}, F_{1}\subset F_{2}$ (or additionally with norm 1 inclusions). Then $E_{1}=E_{2}$ and $F_{1}=F_{2}$ isomorphically (or isometrically).  

\end{proposition}
\begin{proof}
Observe that $E_{2}^{\frac 12}\cdot F_{1}^{\frac 12}\supset E_{1}^{\frac 12}\cdot F_{1}^{\frac 12}=E_{2}^{\frac 12}\cdot F_{2}^{\frac 12}$ implies by the above corollary $F_{1}^{\frac 12}\supset F_{2}^{\frac 12}$. Hence $F_{1}\supset F_{2}$. This implies that $F_{1}$ is isomorphic (or isometric) to $F_{2}$. Similarly the conclusion holds for $E_{1}$ and $E_{2}$, which completes the proof.
\end{proof}     
We will now show how the above results can be used to derive an alternative proof of Lozanovskii's duality theorem (see \cite{Lozanovskii69}, \cite{Lozanovskii73}, \cite{Raynaud97}, \cite{Reisner88}, and \cite{Reisner93} for additional information concerning this duality in the more general case) for Calderon product spaces of Banach function spaces with the Fatou property.
\begin{theorem}\label{multiplier} Let $E$ and $F$ be Banach function spaces with the Fatou property. Assume that $E\cdot F$ is a product Banach function space. Then $F=M(E, E\cdot F)$ isometrically.

\end{theorem}
\begin{proof} Let $f\in B_{E}$ and $g\in F$. Then $\|fg\|_{E\cdot F}\le \|f\|_{E}\|g\|_{F}\le \|g\|_{F}$ shows that $F\subset M(E, E\cdot F)$ and $B_{F}\subset B_{M(E, F)}$. Hence $B_{E}\cdot B_{F}\subset B_{E}\cdot B_{M(E, F)}$. On the other hand by the definition of the norm of $M(E, E\cdot F)$ we have that $B_{E}\cdot B_{M(E, E\cdot F)}\subset B_{E\cdot F}=B_{E}\cdot B_{F}$. Hence $B_{E}\cdot B_{F}=B_{E}\cdot B_{M(E, E\cdot F)}$. This shows that $E\cdot M(E, E\cdot F)$ is a product Banach function space isometric to $E\cdot F$. From the above corollary we conclude that $F=M(E, E\cdot F)$ isometrically.
\end{proof}
	We note that if $E$ is $p$-convex and $F$ is $p^{\prime}$-convex, then the above theorem reproves Theorem 3.5 of \cite{Cwikel03}, which was the main tool used there to prove the uniqueness theorem of the Calderon-Lozaonovskii interpolation method. The next theorem is a special case of Lozanovskii's duality theorem. 

\begin{theorem}\label{dual-p-convex} Let $F$ be a Banach function space  and let $1<p<\infty$. Then $(F^{\frac 1p})^{\prime}=(F^{\prime})^{\frac 1p}L_{1}^{\frac 1{p^{\prime}}}=(F^{\prime})^{\frac 1p}\cdot L_{p^{\prime}}$.

\end{theorem}
\begin{proof}
From Lozanovskii's factorization theorem we have isometrically $F^{\frac 1p}(F^{\frac 1p})^{\prime}=L_{1}$. On the other hand $F^{\frac 1p}\left((F^{\prime})^{\frac 1p}\cdot L_{p^{\prime}}\right)=(F\cdot F^{\prime})^{\frac 1p}\cdot L_{p^{\prime}}=L_{1}^{\frac 1p}\cdot L_{p^{\prime}}=L_{1}$ isometrically. Hence by the above corollary $(F^{\frac 1p})^{\prime}=(F^{\prime})^{\frac 1p}\cdot L_{p^{\prime}}$ isometrically.
Note that both spaces $(F^{\frac 1p})^{\prime}$ and $(F^{\prime})^{\frac 1p}\cdot L_{p^{\prime}}$ have the Fatou property, and by the remark above we do not need the Fatou property for $E$. 
\end{proof}
In the following theorem we provide a simple proof of Lozanovskii's duality theorem for the case of Banach function spaces with the Fatou property.
\begin{theorem} Let  $E$ and $F$ be Banach function spaces with the Fatou property and let $1<p<\infty$. Then
\begin{equation*}
\left(E^{\frac 1p}F^{\frac 1{p^{\prime}}}\right)^{\prime}=(E^{\prime})^{\frac 1p}(F^{\prime})^{\frac 1{p^{\prime}}}.
\end{equation*}
\end{theorem}
\begin{proof}
From Proposition \ref{dual} it follows that $\left(E^{\frac 1p}F^{\frac 1{p^{\prime}}}\right)^{\prime}=M(E^{\frac 1p}, (F^{\frac 1{p^{\prime}}})^{\prime})$, which by the above theorem is equal to $M(E^{\frac 1p}, (F^{\prime})^{\frac 1{p^{\prime}}}\cdot L_{p} )$. Now $M(E^{\frac 1p}, (F^{\prime})^{\frac 1{p^{\prime}}}\cdot L_{p} )=M(E^{\frac 1p}, E^{\frac 1p}(E^{\prime})^{\frac 1p}\cdot(F^{\prime})^{\frac 1{p^{\prime}}} )=M(E^{\frac 1p}, E^{\frac 1p}\cdot(E^{\prime})^{\frac 1p}(F^{\prime})^{\frac 1{p^{\prime}}} )= (E^{\prime})^{\frac 1p}(F^{\prime})^{\frac 1{p^{\prime}}}$ by Theorem \ref{multiplier}.
\end{proof}

\begin{corollary}\label{bidual} Let $F$ be a Banach function space  and let $1<p<\infty$. Then $(F^{\frac 1p})^{\prime\prime}=(F^{\prime\prime})^{\frac 1p}$.
\end{corollary}
\begin{proof} By Theorem \ref{dual-p-convex} we have that $(F^{\frac 1p})^{\prime}=(F^{\prime})^{\frac 1p}L_{1}^{\frac 1{p^{\prime}}}$. Hence by the above theorem $(F^{\frac 1p})^{\prime\prime}=(F^{\prime\prime})^{\frac 1p}(L_{\infty})^{\frac 1{p^{\prime}}}=(F^{\prime\prime})^{\frac 1p}$.

\end{proof}

\section{Factors of Banach function spaces}
Besides considering products of Banach function spaces one can consider the problem of division. We will call the Banach function space $E$ a factor of the Banach function space $G$ if there exist a Banach function space $F$ such that $E\cdot F=G$.  Lozanovskii's factorization theorem says that every Banach function space is a factor of $L_{1}(X, \mu)$. As the example  $\ell_{1}\cdot c_{0}=\ell_{1}\cdot \ell_{\infty}=\ell_{1}$ shows, the space $F$ is not unique, if no additional requirements are imposed. Now $E\cdot F=G$ implies that $F\subset M(E,G)$ and thus $E\cdot M(E,G)=G$. Therefore it is natural to assume that $F=M(E,G)$ and that the norm on $F$ is given by the operator norm of multiplying $E$ into $G$. We will first determine the factors of $L_{p}$ and for that reason we first derive some elementary properties of $M(E, L_{p})$.
\begin{proposition} Let $E$ be a Banach function space. Then $M(E, L_{p})$ is a $p$-convex Banach function space with convexity constant one. 
\end{proposition}
\begin{proof}
Let $0\le f_{k}\in M(E, L_{p})$ for $1\le k\le n$ and let $0\le g\in E$. Then 
\begin{align*}
\left\|\left(\sum_{k=1}^{n}f_{k}^{p}\right)^{\frac 1p}g\right\|_{p}&=\left\| \left(\sum_{k=1}^{n} (f_{k}g)^{p}\right)^{\frac 1p}\right\|_{p}\\
										&=\left(\sum_{k=1}^{n}\|f_{k}g\|_p^p \right)^{\frac 1p}\\
										&\le \left(\sum_{k=1}^{n}\|f_{k}\|_{M(E, L_{p})}^{p}\right)^{\frac 1p}\|g\|_{E}.
\end{align*}
Hence $M(E, L_{p})$ is $p$-convex with convexity constant one.
\end{proof}
In Theorem \ref{dual-p-convex} we saw a description of the K\"othe dual of the $p$-convexification of a Banach function space. In the following theorem we characterize the K\"othe dual of the $p$-concavification of a Banach function space by means of the $p$-concavification of the $p$-convex space $M(E, L_{p})$.
\begin{theorem} Let $E$ be a Banach function space and assume $E$ is $p$-convex for some $p>1$ with convexity constant one. Then the K\"othe dual $(E^{p})^{\prime}$ of the $p$-concavification $E^{p}$ 
of $E$ is isometric with the $p$-concavification $M(E, L_{p})^{p}$ of $M(E, L_{p})$.
\end{theorem}
\begin{proof}
Let $0\le f\in L_{0}$. Then we have
\begin{align*}
\|f\|_{(E^{p})^{\prime}}&=\sup\left( \int f |g|^{p}\,d\mu: \||g|^{p}\|_{E^{p}}\le 1\right)\\
				&=\sup\left(\|\,|f|^{\frac 1p}|g|\,\|_{p}^{p}: \|g\|_{E}\le 1\right)\\
				&=\|f^{\frac 1p}\|_{M(E, L_{p})}^{p}\\
				&=\|f\|_{M(E, L_{p})^{p}}
\end{align*}
This identity proves the theorem.
\end{proof}
\begin{theorem}\label{X^{YY}} Let $E$ be a Banach function space with the Fatou property and let $1<p<\infty$. Then the following are equivalent.
\begin{itemize}
\item[(i)] $E$ is $p$-convex with convexity constant one.
\item[(ii)] $E\cdot M(E, L_{p})$ is a  product Banach function space and $E\cdot M(E, L_{p})=L_{p}$.
\item[(iii)] $M(M(E, L_{p}), L_{p})=E$.
\end{itemize}
\end{theorem}
\begin{proof}
Assume first that $E$  is $p$-convex with convexity constant one. Then by the above theorem $(E^{p})^{\prime}=M(E, L_{p})^{p}$. Thus by Lozanovskii's theorem $E^{p}\cdot M(E,L_{p} )^{p}=L_{1}$. Now taking the $p$-convexication on both sides we see that (ii) holds. Now assume (ii) holds. Then $M(E, L_{p})$ is a $p$-convex Banach function space with the Fatou property and $L_{p^{\prime}}$ is $p^{\prime}$-convex with the Fatou property, so $M(E, L_{p})\cdot L_{p^{\prime}}$ is a product Banach function space with the Fatou property. Now by (ii) we have that $E\cdot (M(E, L_{p})\cdot L_{p^{\prime}})=L_{p}\cdot L_{p^{\prime}}=L_{1}$. Hence $E^{\prime}=M(E, L_{p})\cdot L_{p^{\prime}}$. From this it follows, using Lozanovskii's duality theorem,  that
\begin{align*}
M(M(E, L_{p}), L_{p})&=[\left(M(E, L_{p})^{p}\right)^{\prime}]^{\frac 1p}\\
				&=[\left(M(E, L_{p})^{p}\right)^{\prime}]^{\frac 1p}\left(L_{\infty}\right)^{\frac 1{p^{\prime}}}\\
				&=\left(M(E, L_{p})\cdot L_{p^{\prime}}\right)^{\prime}=E^{\prime\prime}=E.
\end{align*}
Hence (iii) holds. Assume now that (iii) holds. Then $E$ is $p$-convex, since $M(F, L_{p})$ is $p$-convex for any Banach function space $F$. 
\end{proof}
Note that the $p=1$ case of the above theorem is Lozanovskii's factorization theorem, which also shows that as stated we can not drop the assumption that $E$ has the Fatou property. The $p=1$ case and the above proof suggest however that there might be a version without the Fatou property, if we replace in (iii) the space $E$ on the right hand side by $E^{\prime\prime}$. We were not able to establish this in general, but one important special case follows. 

\begin{theorem} Let $E$ be a Banach function space with order continuous norm and let $1<p<\infty$. Then the following are equivalent.
\begin{itemize}
\item[(i)] $E$ is $p$-convex with convexity constant one.
\item[(ii)] $E\cdot M(E, L_{p})$ is a  product Banach function space and $E\cdot M(E, L_{p})=L_{p}$.
\item[(iii)] $M(M(E, L_{p}), L_{p})=E^{\prime\prime}$.
\end{itemize}
\end{theorem}
\begin{proof}
An inspection of the above proof shows that the only step which needs modification is the proof that (iii) implies (i). From (iii) it follows that $E^{\prime\prime}$ is $p$-convex. This implies that $E$ is $p$-convex, since the order continuity of the norm on $E$ implies that $\|f\|_{E}=\|f\|_{E^{\prime\prime}}$ for all $f\in E$.
\end{proof}

As is clear from the above proof the difficulty in establishing the above theorem without any additional assumption on $E$ is to show that if $E^{\prime\prime}$ is $p$-convex, then $E$ is $p$-convex. One can observe that if $E^{\prime\prime}$ is $p$-convex, then $(E^{\prime\prime})^{*}$ is $p^{\prime}$-concave, so also $E^{\prime}=E^{\prime\prime\prime}$ is $p^{\prime}$-concave. The difficulty is then to show that this implies that $E$ is $p$-convex without any additional hypotheses.

\begin{corollary} Let  $E$ and $F$ be Banach function spaces.  Assume that $F=M(E, L_{p})$ and $E=M(F, L_{p})$ for some $1<p<\infty$. Then $E\cdot F$ is a product Banach function space and $E\cdot F=L_{p}$.
\end{corollary}
\begin{proof}
Since $L_{p}$ has the Fatou property it follows that $M(E, L_{p})$ and $M(F, L_{p})$ have the Fatou property. Moreover $E=M(F, L_{p})$ implies that $E$ is $p$-convex and the result follows.
\end{proof}
\begin{remark}
The above corollary can be rephrased as follows. If $K$ and $L$ are unconditional convex bodies in $L_{0}$, which are both maximal for the inclusion $K\cdot L\subset B_{L_{p}}$, then $K\cdot L=B_{L_{p}}$. This result for unconditional convex bodies in $\mathbb R^{n}$ was proved by Bollobas and Leader in \cite{Bollobas93} (see also \cite{Bollobas95}) by a completely different method. We also note that the above Theorem \ref{X^{YY}} provides a partial answer to Question 2 of \cite{Maligranda89}, where it was asked (with a different notation): for which Banach function spaces $E$ and  $G$ with $G\neq L_{1}$ is it true that $M(M(E, G), G)=E$. We will provide a more general answer later on.
 \end{remark}
 One can conjecture more general versions of the above corollary and in fact this was done for finite dimesional unconditional bodies  in \cite{Bollobas93} and \cite{Bollobas95}. One of such conjectures was answered in the negative in Theorem 6 of \cite{Bollobas95}. We will present the answer to another such conjecture, Conjecture 7 of \cite{Bollobas95}, which was already answered in the negative in \cite{Bollobas00}. Our example is essentially the same as the one in \cite{Bollobas00}, but does not use any graph theoretical techniques.
 
 \begin{example}
Let $X=\mathbb R^{3}$ with the counting measure. Define $E=\mathbb R^{3}$ with the norm $\|x\|_{E}=\max\{|x_{1}|+|x_{2}|, |x_{3}|\}$ and $G=\mathbb R^{3}$ with the norm $\|z\|_{G}=|z_{1}|+\max\{|z_{2}|, |z_{3}|\}$. Let $F=M(E, G)$. To compute $\|y\|_{F}$ we shall use the well-known  fact that a convex function attains its maximum on a compact convex set at an extreme point of the compact convex set.
\begin{align*}
\|y\|_{F}&=\max\{|x_{1}y_{1}|+\max\{|x_{2}y_{2}|, |x_{3}y_{3}|\}: |x_{3}|\le 1, |x_{1}|+|x_{2}|\le 1\}\\
	     &=\max\{|x_{1}y_{1}|+\max\{|x_{2}y_{2}|, |y_{3}|\}:  |x_{1}|+|x_{2}|\le 1\}\\
	    &=\max\{|y_{1}|+|y_{3}|, \max\{|y_{2}|, |y_{3}|\}\}=\max\{|y_{1}|+|y_{3}|, |y_{2}|\}.
\end{align*}
Noting the symmetry in the norms of $E$ and $F$ we see immediately that $E=M(F, G)$. We shall show now that $B_{E}\cdot B_{F}$ is not convex, so that $E\cdot F$ is not a product Banach function space. 
Observe that $(1, 0, 0)=(1, 0, 0)\cdot (1,1, 0)\in B_{E}\cdot B_{F}$ and $(0,1,1)=(0, 1, 1)\cdot (0, 1,1)\in B_{E}\cdot B_{F}$. We claim that $(\frac 12, \frac 12, \frac 12) =\frac 12(1, 0, 0)+\frac 12(0, 1, 1)\notin B_{E}\cdot B_{F}$, while it is clear that $(\frac 12, \frac 12, \frac 12)\in B_{G}$. Assume $(\frac 12, \frac 12, \frac 12)=(x_{1}y_{1}, x_{2}y_{2}, x_{3}y_{3})$, where $0\le (x_{1}, x_{2}, x_{3})\in B_{E}$ and $0\le (y_{1}, y_{2}, y_{3})\in B_{F}$. Then $\|x\|_{E}\le 1$ and $\|y\|_{F}\le 1$ imply that $x_{i}\le1$ and $y_{i}\le 1$ for $i=1, 2, 3$. Now $x_{2}y_{2}=\frac 12$ implies that either $x_{2}=1$ and $y_{2}=\frac 12$, or $x_{2}=\frac 12$ and $y_{2}=1$. If $x_{2}=1$, then $x_{1}=0$, contradicting $x_{1}y_{1}=\frac 12$ and if $x_{2}=\frac 12$, then $x_{1}=\frac 12$. so $y_{1}=1$ which implies $y_{3}=0$. contradicting $x_{3}y_{3}=\frac 12$. Hence $(\frac 12, \frac 12, \frac 12)\notin B_{E}\cdot B_{F}$. We also observe also that $B_{G}$ is the convex hull of $B_{E}\cdot B_{F}$. To see this  one can check very easily that the intersections of $B_{E}\cdot B_{F}$ and $B_{G}$ with each of the coordinate planes coincide. Now  the positive part of the ball $B_{G}$ is the convex hull of the intersections of $B_{E^{+}}\cdot B_{F^{+}}$ with each of the coordinate planes and the line segment connecting $(1, 0, 0)$ and $(0, 1,1)$, so $B_{G}$ is the convex hull of $B_{E}\cdot B_{G}$. We illustrate this by the following picture, which shows part of  $B_{E}\cdot B_{F}$.
\end{example}
To get a dual version of Theorem \ref{X^{YY}} we present first a more general duality result.
\begin{theorem}\label{duality} Let $E$ and $F$ be Banach function spaces withe the Fatou property. If $E\cdot F=G$ is a product Banach function space, then $E\cdot G^{\prime}$ is a product Banach function space and $E\cdot G^{\prime}=F^{\prime}$.
\end{theorem}
\begin{proof} It is easy to see that  $G$ is a Banach function space with the Fatou property.  Now $E^{\frac 12}\cdot F^{\frac 12}=G^{\frac 12}$ implies that $E^{\prime\frac 12}F^{\prime\frac 12}=G^{\prime\frac 12}\cdot L_{2}$. Hence $E^{\frac 12}\cdot E^{\prime\frac 12}F^{\prime\frac 12}=E^{\frac 12}\cdot G^{\prime\frac 12}\cdot L_{2}$, which can be written as $L_{2}\cdot F^{\prime\frac 12}=(E^{\frac 12}G^{\prime\frac 12})\cdot L_{2}$. One can check then that $\|\cdot\|_{L_{2}\cdot F^{\prime\frac 12}}=\rho_{(E^{\frac 12}G^{\prime\frac 12})\cdot L_{2}}(\cdot)$, and thus $\rho_{(E^{\frac 12}G^{\prime\frac 12})\cdot L_{2}}$ is a norm. Thus $(E^{\frac 12}G^{\prime\frac 12})\cdot L_{2}$ is a product Banach function space with the Fatou property. Hence $F^{\prime\frac12}=E^{\frac 12}G^{\prime\frac 12}$. This shows that $E^{\frac 12}G^{\prime\frac 12}$ is 2-convex with convexity constant 1 and thus $E\cdot G^{\prime}$ is a product Banach function space. From $F^{\prime\frac12}=E^{\frac 12}G^{\prime\frac 12}$ it follows then that $F^{\prime}=E\cdot G^{\prime}$.

\end{proof}
The following theorem gives a more general sufficient condition for $E\cdot M(E, F)=F$ and $E=M(M(E, F), F)$ to hold. 
\begin{theorem}
Let $E$ and $F$ be Banach function spaces such that there exists $1<p<\infty$ such that $E$ is $p$-convex and $F$ is $p$-concave with convexity and concavity constants equal to 1 and assume $E$ has the Fatou property. Then the following hold.
\begin{itemize}
\item[(i)] $E\cdot M(E, F)$ is a product Banach function space and $E\cdot M(E, F)=F$.
\item[(ii)] $E=M(M(E, F), F)$.
\end{itemize}
\end{theorem}
\begin{proof}
For the proof of (i) observe first that $F^{\prime}$ is $p^{\prime}$-convex, so that $E\cdot F^{\prime}=(E^{p})^{\frac 1p}\left(F^{\prime p^{\prime}}\right)^{\frac 1{p^{\prime}}}$ is a product Banach function space with the Fatou property. Let $G=E\cdot F^{\prime}$. Then by Theorem \ref{duality} we have that $E\cdot G^{\prime}=F^{\prime}$. From the results in section 1 we have that $G^{\prime}=(E\cdot F^{p^{\prime}})^{\prime}=M(E, F^{\prime\prime})=M(E, F)$. Hence $E\cdot M(E, F)=F$. This shows that (i) holds. To prove (ii) note that by (i) we have $M(E. F)\cdot F^{\prime}=E^{\prime}$. Hence $M(M(E, F), F)=(M(E,F)\cdot F^{\prime})^{\prime}=E^{\prime\prime}=E$.
\end{proof}

 \begin{theorem}\label{concavity} Let $F$ be a Banach function space with the Fatou property and let $1<p<\infty$. Then the following are equivalent.
\begin{itemize}
\item[(i)] $F$ is $p$-concave with concavity constant one.
\item[(ii)] $M( L_{p}, F)\cdot L_{p}$ is a  product Banach function space and $M( L_{p}, F)\cdot L_{p}=F$.
\end{itemize}
\end{theorem}
\begin{proof}
Part (i) follows by taking $E=L_{p}$ in the above theorem.  Now assume (ii) holds. Then $F^{\prime}= \left(M( L_{p}, F)\cdot L_{p}\right)^{\prime}=M(M(F^{\prime}, L_{p^{\prime}}), L_{p^{\prime}})$ is $p^{\prime}$-convex, and thus $F$ is $p$-concave.
\end{proof}
As an application of the previous theorems we present a corollary, which reproves Theorem 1 of \cite{Reisner81}. We shall only present the isometric case.
\begin{corollary} Let $1<p<q<\infty$ and let $s$ be defined by $\frac 1s=\frac 1p-\frac 1q$. Then the following are equivalent.
\begin{itemize}
\item[(i)] $E$ is $p$-convex and $q$-concave with convexity and concavity constants equal to 1.
\item[(ii)] $M(L_{q}, E)\cdot M(E, L_{p})$ is a product Banach function space and $L_{s}=M(L_{q}, E)\cdot M(E, L_{p})$.
\end{itemize}
\end{corollary}
\begin{proof}
Assume (i) holds. Then $E$ is $q$-concave for $q<\infty$,  so $E$ has the Fatou property (see \cite{Dodds77}). Hence by the above theorem we have that $L_{q}\cdot M(L_{q}, E)=E$. This implies that $L_{q}^{\frac 12}\cdot M(L_{q}, E)^{\frac 12}\cdot M(E, L_{p})^{\frac 12}=E^{\frac 12}\cdot M(E, L_{p})^{\frac 12}=L_{p}^{\frac 12}$, since $E$ is $p$-convex. On the other hand also by H\"older's inequality $L_{q}^{\frac 12}\cdot L_{s}^{\frac 12}=L_{p}^{\frac 12}$. Hence $M(L_{q}, E)^{\frac 12}\cdot M(E, L_{p})^{\frac 12 }=L_{s}^{\frac 12}=L_{2s}$ by Corollary \ref{cancellation}. This implies that $M(L_{q}, E)^{\frac 12}\cdot M(E, L_{p})^{\frac 12 }$ is $2$-convex, so that $M(L_{q}, E)\cdot M(E, L_{p})$ is a product Banach function space and $L_{s}=M(L_{q}, E)\cdot M(E, L_{p})$. Now assume (ii) holds. We will use then that $L_{q}\cdot L_{p^{\prime}}=L_{s^{\prime}}$ to get that
\begin{equation*}
L_{2}=L_{s}^{\frac 12}L_{s^{\prime}}^{\frac 12}=(L_{q}\cdot M(L_{q}, E))^{\frac 12}\cdot(L_{p^{\prime}}\cdot M(E, L_{p}))^{\frac 12}\subset E^{\frac 12}\cdot (E^{\prime})^{\frac 12}=L_{2}.
\end{equation*}
It follows now from Proposition \ref{uniqueness} that $L_{q}\cdot M(L_{q}, E)=E$ and $L_{p^{\prime}}\cdot M(L_{p^{\prime}}, E^{\prime})=E^{\prime}$. This implies that $E$ is $q$-concave and $E^{\prime}$ is $p^{\prime}$-concave. As $E$ is $q$-concave for $q<\infty$, $E$ has the Fatou property and it follows that $E=E^{\prime\prime}$ is $p$-convex.

\end{proof}
\section{Applications}
In \cite{Bennett96} G. Bennett showed that many classical inequalities involving the $\ell_{p}$-norm can be expressed as product Banach sequence spaces. We will show that some of his results are easy consequences of the results of the previous sections. We start with his result about the spaces $d({\bf a},p)$ and $g({\bf a},p)$. Let ${\bf a}=(a_{1}, a_{2}, \cdots)$ be a non-negative sequence of real numbers with $a_{1}>0$. Define $A_{n}=a_{1}+\cdots a_{n}$. Then for $p\ge 1$ the Banach sequence spaces $d({\bf a},p)$ and $g({\bf a},p)$ are defined as follows:
\begin{equation*}
d({\bf a}, p)=\{{\bf x}: \|{\bf x}\|_{d({\bf a, p})}<\infty\},
\end{equation*}
where
\begin{equation*}
\|{\bf x}\|_{d({\bf a},p)}=\left(\sum_{n=1}^{\infty}a_{n}\sup_{k\ge n}|x_{k}|^{p}\right)^{\frac 1p}
\end{equation*}
and 
\begin{equation*}
g({\bf a}, p)=\{{\bf x}: \|{\bf x}\|_{g({\bf a, p})}<\infty\},
\end{equation*}
where
\begin{equation*}
\|{\bf x}\|_{g({\bf a},p)}=\sup_{n}\left(\frac 1{A_{n}}\sum_{k=1}^{n}|x_{k}|^{p}\right)^{\frac 1p}.
\end{equation*}
\begin{theorem}
Let $1\le p<\infty$. Then  $d({\bf a},p)\cdot g({\bf a},p)$ is a product Banach sequence space and  $d({\bf a},p)\cdot g({\bf a},p)=\ell_{p}$.
\end{theorem}
\begin{proof}
It suffices to prove the theorem for $p=1$. The general result follows then by $p$-convexication. To prove the theorem for $p=1$ we need only show, by Lozanovskii's factorization theorem, that $d({\bf a}, 1)^{\prime}=g({\bf a}, 1)$. For the inequality $\|{\bf y}\|_{d({\bf a}, 1)^{\prime}}\le \|{\bf y}\|_{g({\bf a}, 1)}$ we refer to the first part of the proof of Theorem 3.8 of \cite{Bennett96}. For the reverse inequality observe that
\begin{equation*}
\frac 1{A_{n}}\sum_{k=1}^{n}|y_{k}|\le \|\frac 1{A_{n}}(1,1,\cdots, 1, 0, \cdots)\|_{d({\bf a},1)}\|{\bf y}\|_{d({\bf a}, 1)^{\prime}}=\|{\bf y}\|_{d({\bf a}, 1)^{\prime}}
\end{equation*}
for all $n\ge 1$, which proves $\|{\bf y}\|_{d({\bf a}, 1)^{\prime}}= \|{\bf y}\|_{g({\bf a}, 1)}$.
\end{proof}
\begin{remark}
The above theorem reproves Theorem 3.8 of \cite{Bennett96}. There the factorization was proved directly by a lengthy argument.
\end{remark}
A direct consequence of Theorem \ref{duality} and the above theorem is the following theorem, which corresponds to Theorems 12.3 and 12.22 of \cite{Bennett96}, where again  direct proofs were given.
\begin{theorem}
Let $1<p<\infty$. Then
\begin{equation*}
d({\bf a}, p)^{\prime}=\ell_{p^{\prime}}\cdot g({\bf a}, p),
\end{equation*}
and 
\begin{equation*}
g({\bf a}, p)^{\prime}=\ell_{p^{\prime}}\cdot d({\bf a}, p),
\end{equation*}
where the right hand sides are product Banach sequence spaces.
\end{theorem}

Next we will show how another theorem of \cite{Bennett96} is a direct consequence of the  results of the previous section. In this case we will present the result for arbitrary $\sigma$-finite measure spaces as that will allow us to consider arbitrary order continuous operators. Let $L\subset L_{0}(X, \mu)$ be an order ideal, which has the property that for each measurable set $A$ of positive measure contains a measurable subset $B$ of positive measure such that $\chi_{B}\in L$. Let $1<p<\infty$ and $T:L\to L_{p}$ be a strictly positive order continuous linear map. Then there exists a maximal order ideal $\mathcal D_{p}\subset L_{0}$ such that $T(\mathcal D_{p})\subset L_{p}$. Define $\|f\|_{\mathcal D_{p}}=\|T(|f|)\|_{p}$. Then it is straightforward to see that this defines a norm on $\mathcal D_{p}$ with the Fatou property. Hence $\mathcal D_{p}$ is a Banach function space with respect to the norm $\|\cdot\|_{\mathcal D_{p}}$. 
\begin{theorem}
The Banach function space $\mathcal D_{p}$ is $p$-concave with concavity constant equal to 1. Hence $L_{p}\cdot M(L_{p}, \mathcal D_{p})$ is a product Banach function space and $\mathcal D_{p}=L_{p}\cdot M(L_{p}, \mathcal D_{p})$.
\end{theorem}
\begin{proof}
Let $0\le f_{1}, \cdots, f_{n}\in \mathcal D_{p}$. Let $0\le \alpha_{k}$ such that $\sum_{k=1}^{n} \alpha_{k}^{p^{\prime}}\le 1$.  Then have
\begin{equation*}
\sum_{k=1}^{n}\alpha_{k}Tf_{k}=T\left(\sum_{k=1}^{n}\alpha_{k}f_{k}\right)\le T\left(\left(\sum_{k=1}^{n}f_{k}^{p}\right)^{\frac 1p}\right)  \text{ a.e.}
\end{equation*}
By taking then the supremum over a countable dense set of the positive unit ball of $\ell_{p^{\prime}}(n)$ we get that 
\begin{equation*}
\left(\sum_{k=1}^{n}(Tf_{k})^{p}\right)^{\frac 1p}\le T\left(\left(\sum_{k=1}^{n}f_{k}^{p}\right)^{\frac 1p}\right) \text{ a.e.}
\end{equation*}
Taking $L_{p}$-norms on both sides we get
\begin{equation*}
\left(\sum_{k=1}^{n}\|f_{k}\|_{\mathcal D_{k}}\right)^{\frac 1p}\le \left\|\left(\sum_{k=1}^{n}f_{k}^{p}\right)^{\frac 1p}\right\|_{\mathcal D_{p}},
\end{equation*}
i.e., $\mathcal D_{p}$ is $p$-concave with concavity constant equal to 1. The remaining statements follow now from Theorem \ref{concavity}

\end{proof}
In \cite{Bennett96} the factorization in the  above theorem was proved for Banach sequence spaces as part of Theorem 17.6  by a completely different method, using Maurey's factorization theorem. With essentially the same argument as used above we can extend the above theorem by replacing $L_{p}$ by a $p$-concave  Banach function with concavity constant equal to 1. We leave the details to the reader. 
\bibliography{Bibliography}
\bibliographystyle{amsplain}
\end{document}